\documentclass{amsart}

\usepackage{amsmath, amssymb, amsthm}
\usepackage{hyperref}
\usepackage{cleveref}
\usepackage{chngcntr}
\usepackage{thmtools}
\usepackage{thm-restate}
\usepackage{fancyhdr}
\usepackage{dsfont}
\usepackage[english]{babel}
\usepackage[utf8]{inputenc}
\usepackage{commath}
\usepackage{marvosym}
\usepackage{mathtools}

\usepackage{thmtools}
\usepackage{thm-restate}
\usepackage{soul}

\usepackage{tikz-cd}

\usepackage[margin=1.5in]{geometry}

\usepackage{comment}

\usepackage{enumitem}

\newtheorem{theorem}{Theorem}[section]

\newtheorem{proposition}[theorem]{Proposition}

\newtheorem{lemma}[theorem]{Lemma}

\newtheorem*{theorem*}{Theorem}
\newtheorem*{corollary*}{Corollary}
\newtheorem*{lemma*}{Lemma}
\newtheorem*{proposition*}{Proposition}

\theoremstyle{definition}
\newtheorem{definition}[theorem]{Definition}
\newtheorem{remark}[theorem]{Remark}
\newtheorem{example}[theorem]{Example}

\newtheorem{question}[theorem]{Question}

\newtheorem{chunk}[theorem]{}

\newtheorem{thmx}{Theorem}[section]

\DeclareMathOperator{\tor}{Tor}

\DeclareMathOperator{\uppercidim}{CI^{*}dim}

\newcommand{\D}{\operatorname{D}}
\newcommand{\daq}[4]{\operatorname{D}_{#1}(#2|#3;#4)}
\newcommand{\lco}{\operatorname{L}}

\newcommand{\m}{\mathfrak{m}}
\newcommand{\n}{\mathfrak{n}}
\newcommand{\Char}{\operatorname{char}}

\newcommand{\h}[2][]{\operatorname{H}_{#1}(#2)}

\newcommand{\fdim}[2]{\operatorname{fd}_{#1}#2}
\newcommand{\cone}[1]{\operatorname{cone}(#1)}

\newcommand{\wpi}{\widetilde{\pi}}

\newcommand{\pil}[1]{\operatorname{\pi_{#1}-large}}

\newcommand{\pdim}[2]{\operatorname{pd}_{#1}#2}

\author{Andrew J. Soto Levins}
\address{Texas Tech University, TX 79409. U.S.A.}
\email{ansotole@ttu.edu}
\urladdr{https://sites.google.com/view/andrewjsotolevins}

\author{Ryan Watson}
\address{University of Nebraska Lincon, NE 68588. U.S.A.}
\email{rwatson9@huskers.unl.edu}
\urladdr{https://rawatson1997.github.io}

\title{Large Homomorphisms on the Homotopy Lie Coalgebra}
\date{\today}
\subjclass[2020]{Primary: 13D03, 13D05 Secondary: 16E45}
\keywords{Andr\'e-Quillen homology, homotopy lie algebra, large homomorphism, quasi-complete intersection, upper complete intersection dimension}

\begin{document}
\begin{abstract}
We introduce and study a notion of large homomorphisms on the homotopy lie coalgebra; these homomorphisms are a variant of the large homomorphisms of Levin. As a consequence of our work, we establish new cases of a homotopy lie coalgebra analog of a conjecture of Quillen as proposed by Briggs. 
\end{abstract}

\maketitle
\section*{Introduction}
Throughout this paper, ring means unital commutative noetherian local ring. In 1978 Avramov \cite{Avramov:1978} introduced the notion of small homomorphisms in commutative algebra. These are local homomorphisms $\varphi\colon (R,\m,k) \to (S,\n,k)$ such that the induced map, $\varphi_*\colon \tor^R(k,k) \to \tor^S(k,k)$ is injective. Inspired by this, Levin \cite{Levin:1980} defined the dual notion of large homomorphisms which are local homomorphisms in which the map $\varphi_*$ is surjective. In this paper, we define a notion of large homomorphisms on the homotopy lie coalgebra (see \Cref{C_hlca_def} and \Cref{C_LESwithPi} for the definition of the homotopy lie coalgebra).

Given a local map $\varphi\colon (R,\m,k)\rightarrow (S,\n,\ell)$, we say $\varphi$ is $\pil{n}$ if the natural map
\[\varphi_{n}\colon \pi_{n}(R)\otimes_{k}\ell\rightarrow \pi_{n}(S)\]
is surjective. By \cite[Example 7]{Briggs:2018}, if $\varphi$ is large, then $\varphi$ is $\pil{n}$ for all $n$. This gives many examples of homomorphisms that are $\pil{n}$ for a given $n$; also see \cite{Gheibi/Takahashi:2021}. The main result of this paper is the following, which is a consequence of \Cref{T_UpperPiDimTheorem}. To make sense of it we first recall the notion of upper c.i. dimension. This is a variant of the usual notion of c.i. dimension (see \cite{Avramov/Gasharov/Peeva:1997,Majadas:2016}) but the faithfully flat extension in the quasi-deformation is also required to have regular closed fiber. We denote the upper c.i. dimension of a map by $\uppercidim{\varphi}$.

\begin{thmx} \label{T_MainTheorem} Let $\varphi\colon R\rightarrow (S,\n,\ell)$ be a local homomorphism with $\uppercidim{\varphi}$ finite. If $\varphi$ is $\pil{2n+1}$ for some $n\geq 1$, then $\varphi$ is quasi-complete intersection at $\n$.
\end{thmx}

When $\varphi$ has finite flat dimension, the condition of being $\pil{2n+1}$ is equivalent to $\varphi$ being complete intersection at $\n$ by \Cref{L_PiLargeCI}. In \Cref{E_NontrivialExample} we give classes of surjective large homomorphisms with infinite projective dimension and finite upper complete intersection dimension, and thus providing examples of homomorphisms satisfying the assumptions in \Cref{T_MainTheorem} that are not complete intersection at $\n$. 

In \cite[Theorem 33]{Briggs:2018}, Briggs showed that $\varphi$ is quasi-complete intersection at $\n$ if and only if $\varphi_{n}$ is an isomorphism for $n\geq 3$, and in \cite[Question on page 85]{Briggs:2018} asked must $\varphi$ be quasi-complete intersection at $\n$ if $\varphi_{n}$ is an isomorphism for $n\gg 0$? When $\uppercidim{\varphi}$ is finite, \Cref{T_MainTheorem} says that the answer to Briggs' question is yes. Since $\varphi_{n}$ is an isomorphism for $n\gg 0$ if and only if $\wpi_{n}(\varphi) = 0$ for $n\gg 0$, where $\wpi_{n}(\varphi)$ is a variant of the homotopy lie coalgebra defined in \cite{Briggs:2018}, this question of Briggs can be thought of as a homotopy lie coalgebra analog of a conjecture of Quillen, which we now explain.

Let $\daq{n}{S}{R}{-}$ denote the $nth$ Andr\'e-Quillen homology functor on the category of $S$-modules, defined by Andr\'e \cite{Andre:1971} and Quillen \cite{Quillen:1970}. In \cite[Conjecture 5.6]{Quillen:1970} Quillen conjectured that if $\varphi$ is of finite type, then $\daq{n}{S}{R}{-}=0$ for $n\gg0$ implies $\daq{n}{S}{R}{-}=0$ for $n\geq 3$. This was also conjectured by Avramov in \cite[pg. 459]{Avramov:1999} for maps not necessarily of finite type. In other words, $\varphi$ is quasi-complete intersection when the Andr\'e-Quillen homology functors eventually vanish. We note here that Quillen's conjecture is known to hold when $S$ is an algebra retract of $R$ \cite{Avramov/Iyengar:2003}, when $\varphi$ is onto and $S$ has finite complete intersection dimension over $R$ \cite[Proposition 10]{Soto:2000}, and when $S$ has finite flat dimension over $R$ \cite{Avramov:1999}. Using \cite{Briggs/Iyengar:2020}, it can be shown that if $\uppercidim{\varphi}$ is finite, then $\daq{n}{S}{R}{\ell}$ being zero for some $n\geq 2$ implies $\varphi$ is quasi-complete intersection at $\n$ which establishes the conjecture of Quillen  when $\varphi$ is of finite type and has finite upper c.i. dimension (see \Cref{P_QuillenConjecture}).

\section*{Acknowledgments}
We would like to thank Benjamin Briggs, Daniel McCormick, and Josh Pollitz for the many helpful conversations and suggestions throughout
this project. We would also like to thank Lars Christensen for his comments on an earlier draft.

\section{Background}
In this section, we review necessary background for the paper. Throughout this secion, $\varphi\colon(R,\m,k)\rightarrow (S,\n,\ell)$ is a local homomorphism of noetherian local rings. 
\begin{chunk} 
    Let $\mathfrak{p}$ be a prime ideal of $R$. The ring $k(\mathfrak{p})\otimes_{R}S$ is the fiber of $\varphi$ at $\mathfrak{p}$. The fibers of the completion map $R\rightarrow \widehat{R}$ are called the formal fibers of $R$. The closed fiber of $\varphi$ is $S/\mathfrak{m}S$.
\end{chunk}

\begin{chunk}
    We now recall the definition of a Cohen factorization. A local homomorphism is called \textit{weakly regular} if it is flat with regular closed fiber. By \cite[Theorem 1.1]{Avramov/Foxby/Herzog:1994} the composition of $\varphi$ with the completion map $S\rightarrow \widehat{S}$ admits a \textit{Cohen factorization}
    \[R\xrightarrow{\dot{\varphi}} R'\xrightarrow{\varphi'} \widehat{S},\]
    where $\dot{\varphi}$ is weakly regular, $R'$ complete, and $\varphi'$ surjective. Cohen factorizations reduce the problem of studying ring homomorphisms to only considering surjections and flat homomorphisms with regular closed fiber.
\end{chunk}

\begin{chunk}
    Assume $\widehat R \cong Q /I$ where $Q$ is a regular local ring and $I \subseteq \m_Q^2$, such a $Q$ and $I$ exist by Cohen's structure theorem. A \textit{minimal Cohen model of $R$} is a semi-free extension $Q \to Q[X]$ such that $Q[X] \to \widehat R$ is a surjective quasi-isomorphism and $\partial(\m_{Q[X]}) \subseteq \m_Q+\m_{Q[X]}^2$. These always exist (see \cite[Section 7.2]{Avramov:2010}).

    Let $R \to R' \to \widehat S$ be a Cohen factorization of $\varphi$. By \cite[Section 7.2]{Avramov:2010}, there exists a semi-free extension $R' \to R'[X]$ such that $R'[X] \to \widehat S$ is a surjective quasi-isomorphism and $\partial(\m_{R'[X]}) \subseteq \m_{R'}+\m_{R'[X]}^2$. The factorization $R \to R'[X] \to \widehat S$ is called the \textit{minimal Cohen model for $\varphi$}.
\end{chunk}

\begin{chunk}\label{C_hlca_def}
    Assume $\widehat R \cong Q/I$, where $Q$ is a regular local ring and $I \subseteq \m_Q^2$. Let $Q[X] \to \widehat R$ be a minimal Cohen model. The \textit{homotopy lie coalgebra} of $R$ is defined to be
    \[
         \pi_*(R) := \pi_*(Q[X]) =\Sigma (\m_{Q[X]}/\m^2_{Q[X]}).
    \]
\end{chunk}

\begin{chunk} \label{C_LESwithPi}
    Let $R \to R'[X] \rightarrow \widehat S$ be a minimal Cohen model for $\varphi$. The \textit{homotopy lie coalgebra} of $\varphi$ is defined to be
    \[
        \pi_*(\varphi) := \pi_*(A) = \Sigma(\m_A/\m_A^2)
    \]
    where  $A = k \otimes _R R'[X]$. Moreover, there is an induced map, $\pi_*(R) \otimes_k \ell \to \pi_*(S)$, and following the notation of \cite{Briggs:2018}, we set
    \[
         \wpi_*(\varphi) := H_*\left(\cone{\pi_*(R) \otimes_k \ell \to \pi_*(S)}\right).
    \]
     From this we get a long exact sequence of $\ell$ vector spaces \cite[(3.4)]{Briggs:2018}
     \[
        \cdots \to  \wpi_{n+1}(\varphi) \to \pi_n(R)\otimes_k \ell \to \pi_n(S) \to  \wpi_n(\varphi) \to \pi_{n-1}(R)\otimes_k \ell \to \cdots.
     \]
\end{chunk}

\begin{chunk} \label{c_pi-facts}
    We now collect some facts about $\wpi_*(\varphi)$.
    \begin{enumerate}
        \item There is a natural map $\pi_*(\varphi) \to \wpi_*(\varphi)$, and when $\fdim{R}{S}$ is finite, the map is an isomorphism \cite[Theorem 30]{Briggs:2018}.
        \item If $\fdim{R}{S}$ is finite then, we obtain six term exact sequences \cite[Theorem 25 and (3.2)]{Briggs:2018}
        \[
        \begin{tikzcd}
            0 \ar[r] & \pi_{2i}(R)\otimes_k \ell \ar[r] & \pi_{2i}(S) \ar[r] \ar[d, phantom, ""{coordinate, name=Z}] & \pi_{2i}(\varphi) \ar[dll, rounded corners, to path={--([xshift=2ex]\tikztostart.east) |-(Z) \tikztonodes-|([xshift=-2ex]\tikztotarget.west)--(\tikztotarget)}] \\
            & \pi_{2i-1}(R) \otimes_k \ell \ar[r] & \pi_{2i-1}(S) \ar[r] & \pi_{2i-1}(\varphi) \ar[r] & 0.
        \end{tikzcd}
        \]
        \item By \cite[Theorem 32]{Briggs:2018}, this measures the singularity of $\varphi$ at the maximal ideal of $S$.
        \begin{enumerate}
            \item $\varphi$ is weakly regular if and only if $\wpi_{\geq2}(\varphi)=0$
            \item $\varphi$ is c.i. at $\n$ if and only if $\wpi_{\geq3}(\varphi)=0$
            \item $\varphi$ is q.c.i. at $\n$ (see \Cref{c_qci-def}) if and only if $\wpi_{\geq4}(\varphi)=0$.
        \end{enumerate}
    \end{enumerate}
\end{chunk}

\begin{chunk} \label{C_PiWeaklyRegular}
    If $\dot\varphi \colon (R, \m ,k) \to (R', \n, \ell)$ is weakly regular, then $\pi_n(R) \otimes_k \ell \cong \pi_n(R')$ for all $n \geq 2$. Indeed, $\pi_n(\varphi)=0$ for all $n \geq 2$ by \cite[Proposition 2.6]{Avramov/Iyengar:2003}, so the claim follows by \Cref{c_pi-facts}(2) as $R'$ has finite flat dimension over $R$.
\end{chunk}

\section{Definitions and Basic Properties} \label{Section:DefinitionsAndBasicProperties}
In this section, we define $\pil{n}$ homomorphisms and prove facts we use throughout the paper.
\begin{definition} Let $\varphi\colon (R,\m,k)\rightarrow (S,\n,\ell)$ be a local homomorphism. We say $\varphi$ is $\pil{n}$ if the natural map
\[\pi_{n}(R)\otimes_{k}\ell\rightarrow \pi_{n}(S)\]
is surjective.
\end{definition}

As mentioned in the introduction, every large homomorphism is $\pil{n}$ for all $n$. The following example gives a homomorphism that is not large, but is $\pil{n}$ for almost all $n$.
\begin{example}
    Let $R = k[[x,y]]/(x^3,x^2y)$ and $S= k[[x,y]]/(x^2,y^2)$. The map $\varphi\colon R \twoheadrightarrow S$ is $\pil{n}$ for $n\neq 2$. Indeed, since $S$ is complete intersection, $\pi_n(S)=0$ for $n\geq 3$ (see \cite[Section 7.3 and Theorem 10.2.1]{Avramov:2010}) and thus $\varphi$ is $\pil{n}$ for $n \geq 3$. Lastly, one can check that $\pi_1(R) \to \pi_1(S)$ is an isomorphism, and $\pi_2(R) \to \pi_2(S)$ is the zero map with $\pi_2(S) \neq 0$. Indeed, the map on $\pi_1$ maps between the minimal generators of the maximal ideals of $R$ and $S$ respectively, and the map on $\pi_2$ maps between the minimal generators of the ideals defining $R$ and $S$ respectively. 
\end{example}

\begin{chunk} We now give the definition of locally complete intersection maps \cite{Avramov:1999}. Let $\varphi\colon R\rightarrow S$ be a local homomorphism.
\begin{enumerate}
\item Let
\[R\xrightarrow{\dot{\varphi}} R'\xrightarrow{\varphi'} \widehat{S},\]
be a Cohen factorization of $\varphi$. The map $\varphi$ is \textit{complete intersection} (c.i.) if $\ker{\varphi'}$ is generated by a regular sequence. By \cite[Remark 3.3]{Avramov:1999} this property does not depend on the choice of Cohen factorization.
\item The map $\varphi$ is c.i. at a prime ideal $\mathfrak{q}$ in $S$ if the localization $\varphi_{\mathfrak{q}}\colon R_{\mathfrak{p}}\rightarrow S_{\mathfrak{q}}$ is c.i. where $\mathfrak{p}=\varphi^{-1}(\mathfrak{q})$, and $\varphi$ is \textit{locally c.i.} if $\varphi$ is c.i. at every prime ideal $\mathfrak{q}$ in $S$. 
\end{enumerate}
\end{chunk}

\begin{lemma} \label{L_PiLargeCI} Let $\varphi\colon R\rightarrow (S,\n,\ell)$ be a local homomorphism. Consider the following conditions
\begin{enumerate}
\item $\varphi$ is locally complete intersection.
\item $\fdim{R}{S}<\infty$ and $\varphi$ is $\pil{n}$ for all $n\geq 3$.
\item $\fdim{R}{S}<\infty$ and $\varphi$ is $\pil{2n+1}$ for some $n\geq 1$.
\item $\varphi$ is complete intersection at $\n$.
\end{enumerate}
Then $1\implies 2\iff 3\iff 4$.
\end{lemma}

\begin{proof}
The fact that $(1)$ implies $(4)$ is by definition, so we just need to show the equivalence of $(2)$, $(3)$, and $(4)$. It is clear that $(2)$ implies $(3)$, so we now show $(3)$ implies $(4)$.

Assume $\fdim{R}{S} < \infty$ and $\varphi$ is $\pil{2n+1}$ for some $n \geq 1$. As the flat dimension of $S$ over $R$ is finite, by \Cref{c_pi-facts}, we have exact sequences for all $i \geq 0$,
\[
\begin{tikzcd}  
         \pi_{2i+1}(R) \otimes_k \ell \ar[r] & \pi_{2i+1}(S) \ar[r] & \pi_{2i+1}(\varphi) \ar[r] & 0.
    \end{tikzcd}
\]
As $\varphi$ is $\pil{2n+1}$, we have $\pi_{2n+1}(\varphi)=0$. Hence by \cite[Corollary 5.5]{Avramov/Iyengar:2003}, $\varphi$ is c.i. at $\n$.

Lastly, we show $(4)$ implies $(2)$. As $\varphi$ is c.i. at $\n$, the equality $\pi_n(\varphi)=0$ holds for all $n \geq 3$ by \cite[2.8]{Avramov/Iyengar:2003}. Moreover, since $\varphi$ is c.i., $S$ has finite flat dimension over $R$ by \cite[Lemma 3.2]{Avramov/Foxby/Herzog:1994} so by \Cref{c_pi-facts}, $\varphi$ is $\pil{n}$ for $n \geq 3$.
\end{proof}

\begin{remark}
By \cite[Lemma 5.12.1]{Avramov:1999} when the formal fibers of $R$ are locally c.i., then condition $4$ implies condition $1$ so all four conditions in the lemma are equivalent. However, in \cite[Remarque 3.2.1]{Ferrand/Raynaud:1970} the authors give an example that is c.i. at $\n$ but not locally c.i., so in general the four statements in the lemma are not all equivalent.
\end{remark}

\section{Main Result} \label{Section:UpperAQDim}
In this section we prove \Cref{T_MainTheorem}. We first recall the definition of maps with finite upper c.i. dimension \cite{Majadas:2016} and \cite{Takahashi:2004}.
\begin{chunk} \label{uppercidimension} A local map $\varphi\colon R\rightarrow S$ has \textit{finite upper c.i. dimension}, denoted by $\uppercidim{\varphi}<\infty$, if there exists a Cohen factorization
\[R\rightarrow R'\rightarrow \widehat{S}\]
and a diagram $R'\xrightarrow{h} R''\xleftarrow{f} Q$ of local homomorphisms so that $h$ is weakly regular, $f$ is surjective, $\ker{f}$ is c.i., and $\pdim{Q}{\widehat{S}\otimes_{R'}R''}$ is finite. Such a diagram is called a \textit{quasi-deformation}.  Note that if $\pdim{R'}{\widehat{S}}<\infty$, then $\uppercidim_{R'}\widehat{S}<\infty$. In particular, if $\fdim{R}{S}<\infty$, then $\uppercidim{\varphi}<\infty$ by \cite[Lemma 3.2]{Avramov/Foxby/Herzog:1994}.
\end{chunk}

We next prove the following lemma, which allows us to pass the property of being $\pil{n}$ through quasideformations.

\begin{lemma} \label{L_LongLemma} Let $\varphi\colon R\rightarrow S$ be a local homomorphism. Let
\[R\rightarrow R'\xrightarrow{\varphi'} \widehat{S}\]
be a Cohen factorization of $\varphi$ and assume there is a diagram
\[R'\rightarrow R''\leftarrow Q\]
of local homomorphisms with the first map weakly regular and the second map onto. Consider the following diagram
\[\begin{tikzcd}
	&& {R'} && {\widehat{S}} \\
	Q && {R''} && {\widehat{S}\otimes_{R'}R''}
	\arrow["{\varphi'}", from=1-3, to=1-5]
	\arrow[from=1-3, to=2-3]
	\arrow[from=1-5, to=2-5]
	\arrow["f", from=2-1, to=2-3]
	\arrow["g", from=2-3, to=2-5]
\end{tikzcd}\]
If $f$ is $\pil{n}$ for some $n \geq 2$, then $\varphi$ is $\pil{n}$ if and only if $gf$ is $\pil{n}$.
\end{lemma}

\begin{proof}
Let $k$ be the residue field of $R$, $\ell$ the residue field of $S$, and $L$ the residue field of $Q$. Since $R \to R'$ and $R' \to R''$ are both weakly regular, the composite $R \to R' \to R''$ is also weakly regular \cite[construction 4.4]{Avramov/Foxby:1998}. Similarly, $S \to \widehat S \otimes_{R'} R''$ is weakly regular as well. Consider the following diagram induced by the functoriality of $\pi_n(-)$,
\[
\begin{tikzcd}
    & \pi_n(R) \otimes_k L \ar[r] \ar[d] & \pi_n(S) \otimes_{\ell} L \ar[d]\\
    \pi_n(Q) \ar[r]  & \pi_n(R'') \ar[r] & \pi_n(\widehat S \otimes_{R'}R'').
\end{tikzcd}
\]

Since $R \to R''$ and $S \to \widehat S \otimes_{R'}R''$ are weakly regular, the vertical maps are isomorphisms for all $n \geq 2$ by \Cref{C_PiWeaklyRegular}. Hence, $\varphi$ is $\pil{n}$ if and only if $g$ is $\pil{n}$. Assuming $f$ is $\pil{n}$, it follows that $gf$ is $\pil{n}$ if and only if $\varphi$ is $\pil{n}$.
\end{proof}

Before stating the theorem, we recall some definitions.
\begin{chunk} André-Quillen homology was introduced independently by André \cite{Andre:1971} and Quillen \cite{Quillen:1970}. For a ring homomorphism $\varphi\colon R\rightarrow S$ one can attach to it a complex $\lco_{S|R}$, known as the cotangent complex, and given an $S$-module $M$ the $n$th André-Quillen homology module of $S$ over $R$ with coefficients in $M$ is
\[\daq{n}{S}{R}{M} := \h[n]{\lco_{S|R}\otimes_{S}M}.\]
\end{chunk}

\begin{chunk}\label{c_qci-def} We now give the definition of \textit{quasi-complete intersection} (q.c.i) maps. See \cite[Section 1]{Avramov/Henriques/Sega:2013} for a discussion of q.c.i ideals. For an ideal $I$ in a noetherian local ring $R$, the ideal $I$ is q.c.i. if and only if $\D_{n}(R/I|R;-)=0$ for $n\geq 3$ by \cite[Corollary 3']{Blanco/Majadas/Rodicio:1998}. Let $\varphi\colon R\rightarrow (S,\mathfrak{n},\ell)$ be a local homomorphism. The map $\varphi$ is q.c.i at $\mathfrak{n}$ if in some Cohen factorization
\[R\rightarrow R'\xrightarrow{\varphi'} \widehat{S}\]
the ideal $\ker{\varphi'}$ is q.c.i. By \cite[Lemma 7.2]{Avramov/Henriques/Sega:2013} this property does not depend on the choice of Cohen factorization. 
\end{chunk}

\begin{theorem} \label{T_UpperPiDimTheorem} Let $\varphi\colon R\rightarrow (S,\n,\ell)$ be a local homomorphism with $\uppercidim{\varphi}$ finite. If $\varphi$ is $\pil{2n+1}$ for some $n\geq 1$, then $\varphi$ is quasi-complete intersection at $\n$.
\end{theorem}

\begin{proof}
Consider the following diagram witnessing the finiteness of $\uppercidim{\varphi}$:
\[\begin{tikzcd}
	&& {R'} && {\widehat{S}} \\
	Q && {R''} && {\widehat{S}\otimes_{R'}R''.}
	\arrow["{\varphi'}", from=1-3, to=1-5]
	\arrow["h", from=1-3, to=2-3]
	\arrow[from=1-5, to=2-5]
	\arrow["f", from=2-1, to=2-3]
	\arrow["g", from=2-3, to=2-5]
\end{tikzcd}\]
By \Cref{c_pi-facts}, $\uppercidim{\varphi}<\infty$ implies $\wpi_{n}(f)=0$ for $n\geq 3$. Thus, $f$ is $\pil{n}$ for $n \geq 3$ by \Cref{C_LESwithPi}. Hence, $gf$ is c.i. by \Cref{L_PiLargeCI} and \Cref{L_LongLemma}. Therefore, if $\uppercidim{\varphi} < \infty$, then $g$ is q.c.i. by \cite[8.9.1]{Avramov/Henriques/Sega:2013}, which implies that $\varphi$ is q.c.i. at $\n$ by \cite[8.7]{Avramov/Henriques/Sega:2013}.
\end{proof}

We should point out the following.
\begin{remark} Let $(R,\m,k)$ be a noetherian local ring, $I$ be an ideal, and $\varphi\colon R\rightarrow R/I$ be the quotient map. Consider the following conditions.
\begin{enumerate}
    \item $\varphi$ is q.c.i. and $I\cap \m^{2}=\m I$.
    \item $\varphi$ is large.
\end{enumerate}
Then 1 implies 2, and 2 implies 1 if $\uppercidim{\varphi}$ is finite.

The first implication is by \cite[Theorem 6.2]{Avramov/Henriques/Sega:2013} and \cite[Example 2.4.4]{Gheibi/Takahashi:2021}. If $\uppercidim{\varphi}$ is finite and 2 holds, then 1 holds by \Cref{T_UpperPiDimTheorem} and \cite[2.3]{Gheibi/Takahashi:2021}.
\end{remark}

\Cref{E_NontrivialExample} gives classes of surjective large homomorphisms with infinite projective dimension and finite upper complete intersection dimension, and thus providing examples of homomorphisms satisfying the assumptions in \Cref{T_UpperPiDimTheorem} that are not c.i. at $\n$. As mentioned in the introduction by \cite[Example 7]{Briggs:2018}, if a homomorphism is large, then it is $\pil{n}$ for all $n$. Before getting to the example, we first need the following lemma.
\begin{lemma} \label{L_LemmaForNontrivialExample} Let $(R,\m,k)$ be a local complete intersection and $I$ a nonzero ideal so that $I\cap \m^{2}=\m I$. If $I$ is not generated by a regular sequence, then $R/I$ has infinite projective dimension and $\uppercidim{\varphi}$ is finite, where $\varphi$ is the quotient map $R\rightarrow R/I$.
\end{lemma}

\begin{proof}
If $R/I$ has finite projective dimension, then by the proof of \cite[Example 2.4.1]{Gheibi/Takahashi:2021} and the equality $I\cap \m^{2}=\m I$ implies $I$ is generated by a regular sequence. To see that $\uppercidim{\varphi}$ is finite, notice that $\widehat{R}$ is a quotient of a regular local ring $Q$ by a regular sequence. Since 
\[R\rightarrow \widehat{R}\rightarrow \widehat{R/I}\]
is a Cohen factorization of $\varphi$ and since 
\[R\rightarrow \widehat{R}\leftarrow Q\]
is a quasideformation, $\uppercidim{\varphi}$ is finite.
\end{proof}

\begin{example} \label{E_NontrivialExample} Let $(R,\m,k)$ be a complete intersection and $I$ a nonzero ideal so that $I\cap \m^{2}=\m I$. In other words, a minimal generating set for $I$ can be extended to a generating set for $\m$. Note that by \cite[2.3]{Gheibi/Takahashi:2021} the equality $I\cap \m^{2}=\m I$ is a necessary condition for the quotient map $\varphi\colon R\rightarrow R/I$ to be large. 
\begin{enumerate}
\item Assume $R$ is not regular. By \cite[Example 2.4.3]{Gheibi/Takahashi:2021} the quotient map $R\rightarrow k$ is large. By \cite[Theorem 2.7]{Assmus:1959} and \cite[Theorem 2.3.11]{Bruns/Herzog:1998} $R$ is a complete intersection if and only if $\mathfrak{m}$ is a q.c.i. ideal, and so the quotient map $R\rightarrow k$ is q.c.i.. Also, $R/I$ has infinite projective dimension and $\uppercidim{\varphi}$ is finite by \Cref{L_LemmaForNontrivialExample}.
\item If $I$ is contained in the socle of $R$, then $\varphi$ is large by \cite[Example 2.4.2]{Gheibi/Takahashi:2021}. Also, since $I$ is contained in the socle, $I$ cannot be generated by regular sequence, and so $R/I$ has infinite projective dimension and $\uppercidim{\varphi}$ is finite by \Cref{L_LemmaForNontrivialExample}.  
\item If $I$ is a summand of $\m$ that is not generated by a regular sequence, then $\varphi$ is large by \cite[Example 2.4.6]{Gheibi/Takahashi:2021}. Also, $R/I$ has infinite projective dimension and $\uppercidim{\varphi}$ is finite by \Cref{L_LemmaForNontrivialExample}. The ideal $I=(x)$ in the ring $R=k[[x,y]]/(xy)$ satisfies these conditions. Note that if $\m$ is decomposable, then $R$ is forced to be a one dimensional hypersurface by \cite[Corollary 2.7]{Nasseh/SatherWagstaff/Takahashi/VandeBogert:2019}. Local rings with decomposable maximal ideal are known as \textit{fiber product rings}, see \cite[Fact 2.1]{Nasseh/SatherWagstaff/Takahashi/VandeBogert:2019}.
\item Assume $I$ is not generated by a regular sequence and there is a local homomorphism $h\colon R/I\rightarrow R$ so that $\varphi h$ is the identity on $R/I$. Then $\varphi$ is large by the natural change of ring homomorphisms on $\tor_{(-)}(k,k)$. 
Also, $R/I$ has infinite projective dimension and $\uppercidim{\varphi}$ is finite by \Cref{L_LemmaForNontrivialExample}.
\end{enumerate}
\end{example}

We end this paper by showing how the work above relates to Andr\'e-Quillen homology.

\begin{chunk}
    When $\Char{\ell}=0$, there is an isomorphism
    \[\pi_{n+1}(\varphi)\cong \daq{n}{S}{R}{\ell}\]
    for all $n$ by \cite{Quillen:1970, Avramov/Halperin:1987}, and so one can use \Cref{T_UpperPiDimTheorem} to recover the following proposition in this setting. The following is a consequence of \cite[Theorem A]{Briggs/Iyengar:2020}.  
\end{chunk}

\begin{proposition} \label{P_QuillenConjecture} Let $\varphi\colon R\rightarrow (S,\n,\ell)$ be a local homomorphism with $\uppercidim{\varphi}$ finite. If $\daq{n}{S}{R}{\ell}=0$ for some $n\geq 2$, then $\varphi$ is q.c.i. at $\n$.
\end{proposition}

\begin{proof}
Let 
\[\begin{tikzcd}
	&& {R'} && {\widehat{S}} \\
	Q && {R''} && {\widehat{S}\otimes_{R'}R''}
	\arrow["{\varphi'}", from=1-3, to=1-5]
	\arrow[from=1-3, to=2-3]
	\arrow[from=1-5, to=2-5]
	\arrow["f", from=2-1, to=2-3]
	\arrow["g", from=2-3, to=2-5]
\end{tikzcd}\]
be a diagram witnessing the finiteness of the upper complete intersection dimension of $\varphi$ and let $L$ be the residue field of $\widehat{S}\otimes_{R'}R''$. Since $\ell$ is a subfield of $L$, $\daq{n}{S}{R}{L}$ is a direct sum of $\daq{n}{S}{R}{\ell}$, giving the equality below
\[\daq{n}{\widehat{S}\otimes_{R'}R''}{R''}{L} \cong \daq{n}{\widehat{S}}{R'}{L} \cong \daq{n}{S}{R}{L} = 0.\]
The first isomorphism is by \cite[6.3]{Iyengar:2007} and the second isomorphism is by \cite[Lemma 1.7]{Avramov:1999}. Since $f$ is c.i., $\daq{i}{R''}{Q}{L}=0$ for $i\geq 2$, and so from the Jacobi-Zariski exact sequence \cite[Theorem 5.1]{Andre:1974}
\[\daq{n}{R''}{Q}{L}\rightarrow \daq{n}{\widehat{S}\otimes_{R'}R''}{Q}{L}\rightarrow \daq{n}{\widehat{S}\otimes_{R'}R''}{R''}{L}\]
we get $\daq{n}{\widehat{S}\otimes_{R'}R''}{Q}{L} = 0$. The implies $gf$ is c.i. by \cite[Theorem A]{Briggs/Iyengar:2020} and \cite[Proposition 4.57]{Andre:1974}. Now by \cite[8.9]{Avramov/Henriques/Sega:2013} $g$ is q.c.i., which implies $\varphi$ is q.c.i. at $\n$ by \cite[8.7]{Avramov/Henriques/Sega:2013}.
\end{proof}

It is unknown to the authors whether the assumption that the map has finite upper c.i. dimension is necessary or if it can be weakened to having finite c.i. dimension. Thus, we ask the following question.

\begin{question}
    Do \Cref{T_UpperPiDimTheorem} and \Cref{P_QuillenConjecture} still hold under the assumption that $\varphi \colon R \to S$ has finite c.i. dimension instead of finite upper c.i. dimension?
\end{question}

\bibliographystyle{alpha}
\bibliography{references}
\end{document}